\documentclass{amsart}
\usepackage{cite}
\usepackage{amsmath, amsthm, amssymb,tikz}
\usepackage{accents}
\usepackage{csquotes}
\usepackage{listings}  
\usepackage{tikz-cd}
\usepackage{young,ytableau}
\usepackage[vcentermath]{youngtab}
\usepackage{MnSymbol}
\usepackage[hidelinks]{hyperref}
\usetikzlibrary{arrows}
\usepackage{longtable}
\usepackage{mathtools}
\usepackage{cellspace}
\usepackage{bbm}
\numberwithin{equation}{section}
\newtheorem{lemma}{Lemma}[section]
\newtheorem{theorem}[lemma]{Theorem}
\newtheorem{conjecture}[lemma]{Conjecture}
\newtheorem{corollary}[lemma]{Corollary}
\newtheorem{proposition}[lemma]{Proposition}
\theoremstyle{definition}
\newtheorem{example}[lemma]{Example}
\newtheorem{remark}[lemma]{Remark}
\newtheorem{definition}[lemma]{Definition}

\def\16{{\bf 16}}
\def\1{{\bf 1}}
\def\2{{\bf 2}}
\def\3{{\bf 3}}
\def\4{{\bf 4}}

\def\tilde{\widetilde}

\def\Sym{{\mathrm{Sym}}}

\DeclareMathOperator{\Hilb}{Hilb}
\DeclareMathOperator{\Spec}{Spec}

\DeclareMathOperator{\Proj}{\mathrm{Proj}}

\DeclareMathOperator{\gr}{gr}

\newcommand{\cO}{\mathcal{O}}
\newcommand{\cF}{\mathcal{F}}

\newcommand{\cR}{\mathcal{R}}

\newcommand{\cN}{\mathcal{N}}

\newcommand{\fs}{\mathfrak{s}}
\newcommand{\fo}{\mathfrak{o}}
\newcommand{\fp}{\mathfrak{p}}

\newcommand{\fb}{\mathfrak{b}}
\newcommand{\fm}{\mathfrak{m}}

\newcommand{\fl}{\mathfrak{l}}

\newcommand{\fg}{\mathfrak{g}}
\newcommand{\ft}{\mathfrak{t}}

\newcommand{\C}{{\mathbb C}}
\newcommand{\R}{{\mathbb R}}
\newcommand{\N}{{\mathbb N}}

\newcommand{\Z}{{\mathbb Z}}

\newcommand{\Q}{{\mathbb Q}}

\newcommand{\CM}{\mathcal{CM}}
\newcommand{\Diff}{\mathcal{D}}

\newcommand{\e}{\mathbf{e}}

\begin{document}
\title{A Lie-theoretic generalization of some Hilbert schemes}
\author{Oscar Kivinen}
\address{Aalto University}
\begin{abstract}
    We define several versions of a class of varieties $X_{\fg}$ attached to a complex reductive Lie algebra $\fg$, generalizing the Hilbert scheme of points on the plane. These include trigonometric and elliptic versions attached to the corresponding groups. We also define the corresponding isospectral varieties $Y_{\fg}$.  
    
    We prove a Gordon--Stafford localization theorem for $X_\fg$ and the corresponding equal-parameter rational Cherednik algebras, relate these varieties to the affine Springer fiber--sheaf correspondence of \cite{GKO}, and discuss examples. 
    
    We conjecture that the torus-fixed points of our varieties are in bijection with two-sided cells in the finite Weyl group and prove this in types $ABC$. We relate these results to known results about Calogero--Moser spaces.

\end{abstract}
\maketitle 
\setcounter{tocdepth}{1}
\tableofcontents
\section{Introduction}
Recent conjectures in knot theory \cite{OR, GNR} and affine Springer theory \cite{GKO, KT} suggest that there should be the following commutative triangle:
\begin{center}
\begin{figure}
\label{fig:triangle}
\begin{tikzcd}[arrows=Rightarrow]
    &  \fbox{\begin{tabular}{@{}c@{}}Affine Springer fibers\\ for $GL_n$\end{tabular}} \arrow[dr,"\text{\cite{GKO,KT}}"] \arrow[dl,"\text{\cite{ORS}}"'] &  \\
  \fbox{\begin{tabular}{@{}c@{}}HOMFLY-type\\ braid invariants\end{tabular}}   \arrow[rr,"\text{\cite{GNR,OR,HoLi}}"]
    &   & \fbox{\begin{tabular}{@{}c@{}}Coherent sheaves\\ on $\Hilb^n(\C^2)$\end{tabular}} 
\end{tikzcd}
\caption{}
\end{figure}
\end{center}

The arrows in this diagram have been explained in various sources; see, e.g., \cite{GKS} for a recent survey.
The main theme of this article is to provide a replacement for $\Hilb^n(\C^2)$ for other reductive Lie algebras, with the hope of generalizing the lower right vertex.

Consider the other vertices. 
For the lower left, given a conjugacy class of braids on $n$ strands, the HOMFLY-PT polynomial is a two-variable invariant of the braid closure in $S^3$. 
This invariant has been categorified to the HOMFLY or triply-graded homology of links in $S^3$. The latter can be defined via Hochschild homology of Rouquier complexes of Soergel bimodules for $GL_n$. Fixing $n$, both invariants have annular generalizations, yielding invariants of braid closures in a solid torus, or equivalently of conjugacy classes of braids in $Br_n=\pi_1(\C^n\backslash\{\text{diagonals}\}/S_n)$.  
 
The configuration space of $n$ points in $\C$ can be thought of as the regular part of a Cartan subalgebra $\ft\subset \fg\fl_n(\C)$. Replacing $\fg\fl_n$ by another reductive Lie algebra $\fg$, we consider the braid group $Br_W=\pi_1(\ft^{reg}/W)$. 
Similar HOMFLY--like (annular) invariants for conjugacy classes in $Br_W$ can be constructed using Markov traces in the polynomial case, or Hochschild homology of Rouquier complexes in the categorified setting \cite{WW, Trinh}. The resulting invariants depend only on the underlying root system. 

The top vertex of the triangle also generalizes to other reductive $G$ by considering affine Springer fibers for these groups. Now there is a dependence on more data than simply the root system, which gives rise to various complications in establishing the mostly conjectural left and right arrows. We suppress the additional complications arising from these choices and do not discuss affine Springer fibers in detail.

Finally, we address the lower right vertex: is there a replacement for $\Hilb^n(\C^2)$ for other $\fg$, coherent sheaves on which encode information about homology of affine Springer fibers or braid invariants? 
We define such replacement varieties $X_\fg$ for arbitrary reductive Lie algebras, such that $X_{\fg\fl_n}\cong \Hilb^n(\C^2)$. These occur in three versions: $X_{\fg, sgn}, X_{\fg,diag}$ and $X_{\fg,symb}$. The symbolic version $X_{\fg, symb}$ appears most relevant for applications.

We note that there are trigonometric and elliptic generalizations attached to reductive groups. Both generalizations arise naturally in the context of Coulomb branches of SUSY gauge theories, constructed mathematically in \cite{BFN}. For ease of exposition, we mostly stay in the rational case, with the exception of Section \ref{sec:coulomb}.

It is no surprise that the geometry of $X_\fg$ is strongly related to the representation theory of rational Cherednik algebras. In particular, we conjecture that $X_{\fg,symb}$ is a hyper-Kähler rotation of the Calogero--Moser space attached to the center of the rational Cherednik algebra for $\fg$ at $t=0$, and we prove this in types $ABC$. As stated in Theorem \ref{thm:zalgebra}(2), in general $X_{\fg,symb}$ is the commutative degeneration of a $\Z$-algebra attached to the rational Cherednik algebra with equal parameters and $t=1$.

We remark that the bottom vertices of Figure \ref{fig:triangle} also generalize to other Coxeter groups and their reflection representations. Many notions in Section \ref{sec:cherednik} even generalize to complex reflection groups, although we do not consider these cases.

The plan of the paper is as follows. We first recall some of Haiman's results on the Hilbert scheme of points on $\C^2$ in Section \ref{sec:hilb}, and then define the three different versions of $X_\fg$ in Section \ref{sec:definitions}. 
In Section \ref{sec:cherednik} we prove a localization theorem for these varieties. 
In Section \ref{sec:geomprop}, we pose some conjectures on the geometry of $X_\fg$ and prove most of these conjectures in types $ABC$ in Appendix \ref{app:proofs}.

\section{Hilbert schemes}
\label{sec:hilb}
In this section, we recall some results on the Hilbert scheme of points on the plane, $\Hilb^n(\C^2)$.
Recall that by definition, $\Hilb^n(\C^2)$ parametrizes subschemes of length $n$ in the plane $\C^2$.     On closed points we may write 
    $$\Hilb^n(\C^2)=\{I\subseteq \C[x,y]|\dim_\C\C[x,y]/I=n\}.$$
    There is also a well-known quiver variety/ADHM description, realizing $\Hilb^n(\C^2)$ as the quotient by $GL_n$ of the space of triples $$\{(x,y,v)\in \fg\fl_n\times\fg\fl_n\times \C^n| [x,y]=0, \C[x,y]v=\C^n\}.$$ Instead of these two descriptions, we seek to generalize the following Proj construction by Haiman in \cite{Hai1}.
\begin{theorem}
\label{thm:haimanhilb}
$$\Hilb^n(\C^2)=\Proj\bigoplus_{d\geq 0} (\Delta A)^d,$$ where $\Delta A\subseteq \C[x_1,\ldots,x_n, y_1,\ldots,y_n]^{S_n}$ is the ideal consisting of products of the form $\Delta f$, where $\Delta=\prod_{i<j}(x_i-x_j)$ and $f\in A:=\C[x_1,\ldots, x_n,y_1,\ldots,y_n]^{sgn}$ is an alternating polynomial for the diagonal action.
\end{theorem}
\begin{proof} (Sketch.)
    On $\Spec \C[x_1,\ldots,x_n, y_1,\ldots, y_n]^{S_n}=\Sym^n \C^2$ there is an open locus $V=(\Sym^n\C^2)^{reg}$ of collections of $n$ points where all the points are distinct. Restricted to $V$, the Hilbert-Chow map $\Hilb^n(\C^2)\to \Sym^n \C^2$ is an isomorphism.
    For every partition $\mu\vdash n$, there are open charts $U_\mu \subseteq \Hilb^n(\C^2)$ consisting of those colength $n$ ideals $I\subseteq \C[x,y]$ for which $\C[x,y]/I$ is spanned by the monomials corresponding to the diagram of $\mu \subset \N\times \N$. On $U_\mu\cap V$, one can construct regular functions $\Delta_D/\Delta_\mu$, where $\Delta_D:=\det(x_i^{p_j}y_i^{q_j})_{i,j=1}^n$ are alternating polynomials attached to $n$-element subsets $D=\{(p_1,q_1),\ldots, (p_n,q_n)\}\subset \N\times \N$. One proves that these regular functions glue as we vary $\mu$ and also extend to regular functions on all of $\Hilb^n(\C^2)$. The universal property of blowing up then implies that there is a surjective map $\alpha: \Hilb^n(\C^2)\to \Proj\bigoplus_{d\geq 0} (\Delta A)^d$. Further, one checks that the pullbacks of the above $\Delta_D/\Delta_\mu$ give all the regular functions on $U_\mu$, showing that the map on structure sheaves is also surjective and $\alpha$ must be an isomorphism.
\end{proof}

Let us also recall an easy corollary of this result.
\begin{corollary}
\label{cor:isospec}
    The isospectral Hilbert scheme, i.e. the reduced fiber product 
    $$\begin{tikzcd}
        Y_n \arrow[d]\arrow[r] & \C^{2n}\arrow[d]\\
        \Hilb^n(\C^2)\arrow[r] & \Sym^n\C^2
    \end{tikzcd}$$
    is isomorphic to $$\Proj \bigoplus_{d\geq 0} (\Delta J)^d,$$ where $J=\C[x_1,\ldots, y_n]\cdot A$ is the ideal in $\C[x_1,\ldots, y_n]$ generated by the diagonally alternating polynomials.
\end{corollary}
\begin{remark}
One might wonder why we use $\Delta A$. Of course, one could take the ideal $fA$, where $f$ is any alternating polynomial, and get an isomorphic Proj. After a further Veronese twist, one could also consider the ideal $A^2$ as in \cite{Hai1}.
\end{remark}
\begin{remark}
\label{remark:haiman}
We will use this theorem as an inroad to defining $X_{\fg,sgn}$ in the next section. As Haiman remarks in \cite{Hai1}, since an alternating polynomial in a single set of variables vanishes on the diagonals, it is natural to expect that the radical of the ideal $$A^2\subseteq \C[x_1,\ldots,x_n,y_1,\ldots,y_n]^{S_n}$$ coincides with the ideal of the locus in $\Sym^n \C^2$ where two or more points coincide. If two points coincide in $\C^{2}$, their $x$- and $y$-coordinates have to be equal. Therefore, the ideal of this locus can be written as $\e I$, where $\e=\frac{1}{n!}\sum_{\sigma\in S_n}\sigma$ and $I$ is the ideal $I=\bigcap_{i<j}\langle x_i-x_j,y_i-y_j\rangle\subseteq\C[x_1,\ldots,x_n,y_1,\ldots,y_n]$.
Lifting these considerations to $Y_n$, one can similarly ask whether the ideal $J$ from Corollary \ref{cor:isospec} equals $I$.

Knowing the analogous equality of ideals in dimension $1$ is almost trivial, so it might seem unsurprising that indeed we have \begin{equation}
\label{eq:haiman}
I=J.\end{equation} However, the only known proof of Eq. \eqref{eq:haiman} is a corollary of the Polygraph Theorem of \cite{Haimanpolygraph}, and in fact this equality is equivalent to the Cohen-Macaulay/Gorenstein properties of the isospectral Hilbert scheme $Y_n$. 
\end{remark}

\section{The analogs for other Lie algebras}
\label{sec:definitions}
Let $G$ be a complex reductive group. We use standard notation such as $T, \ft$ for maximal tori and Cartan subalgebras and $B, \fb$ for Borel subgroups and subalgebras. Just as we replaced $\C^n\backslash\{\text{diagonals}\}$ by $\ft^{reg}$ in the introduction, we may replace the ring $$\C[x_1,\ldots,x_n,y_1,\ldots,y_n]$$ and its diagonal $S_n$-action by the ring $\C[\ft\oplus \ft^*]$ of regular functions on $\ft\oplus \ft^*$ and its diagonal $W$-action, where $W$ is the Weyl group of $G$.  

Based on the previous section, a natural first attempt at defining $X_\fg$ for an arbitrary reductive $\fg$ is the following.
\begin{definition}
Let $A=\C[\ft\oplus \ft^*]^{sgn}\subseteq \C[\ft\oplus\ft^*]$ be the space of diagonally alternating polynomials for the $W$-action.  Let $\Delta=\prod_{\alpha\in \Phi^+} \alpha$ and define
 $$X_{\fg,sgn}:=\Proj\bigoplus_{d\geq 0} (\Delta A)^d$$
\end{definition}
In other words, $X_{\fg,sgn}$ is the blow-up of $\C[\ft\oplus \ft^*]^W$ in the ideal $\Delta A$, where $\Delta$ and $A$ are as above. Note that we have dropped the subscript $\fg$ in the commutative algebraic objects, hoping this causes no confusion.

Similarly, we can consider the ideal $J$ in $\C[\ft\oplus \ft^*]$ generated by $A$. This defines the {\em isospectral variety} of $X_{\fg,sgn}$:
\begin{definition}
    The isospectral variety $Y_{\fg,sgn}$ of $X_{\fg,sgn}$ is $$\Proj\bigoplus_{d\geq 0} (\Delta J)^d$$
\end{definition}
Since the $W$-action is bigraded, the construction of $X_{\fg,sgn}$ and $Y_{\fg,sgn}$ respects the natural $(\C^*)^2$-action on $\ft\oplus \ft^*$. Conjectures of Bonnaf\'e suggest that these varieties have desirable properties; for example, 
\begin{conjecture}
\label{conj:bonnafe}
    There is a bijection 
    $$\{\text{Two-sided cells in } W\}\leftrightarrow X_{\fg,sgn}^{\C^\times}.$$
\end{conjecture}
In particular, there are only finitely many fixed points, depending only on $W$.
For $\fg=\fg\fl_n$, $W=S_n$, and this is the bijection between fixed points on the Hilbert scheme and partitions. Bonnaf\'e has also verified Conjecture \ref{conj:bonnafe} for type $B_2$ and some non-Weyl group versions. 

Following Remark \ref{remark:haiman}, let $I\subseteq \C[\ft\oplus \ft^*]$ be the ideal $I=\bigcap_{\alpha\in \Phi^+} \langle \alpha^\vee, \alpha\rangle$, and consider $\e\Delta I\subseteq \C[\ft\oplus \ft^*]^W$, where $\e:=|W|^{-1}\sum_{w\in W}w$. Blowing up $\ft\oplus \ft^*/W$ in the union of codimension two root hyperplanes $\{\alpha=0\}\times \{\alpha^\vee=0\}$ gives
\begin{definition}
    $$X_{\fg,diag}=\Proj \bigoplus_{d\geq 0} \e(\Delta I)^d$$ 
\end{definition}
Similarly, one can define the isospectral variety as 
\begin{definition}
$$Y_{\fg,diag}:=\Proj\bigoplus_{d\geq 0} (\Delta I)^d$$
\end{definition}
Again, since a diagonally alternating polynomial vanishes on the codimension 2 root hyperplanes, there are inclusions $J\subseteq I$ and $\Delta A\subseteq {\bf e} \Delta I$, and in particular maps $Y_{\fg,diag}\to Y_{\fg,sgn}$ and $X_{\fg,diag}\to X_{\fg,sgn}$.

We conjecture that the map $X_{\fg,diag}\to X_{\fg,sgn}$ is an isomorphism. However, we check in Appendix \ref{sec:examples} that the map between isospectral varieties is not an isomorphism in type $B_3$. Thus $Y_{\fg,diag}$ is not the reduced fiber product of $X_{\fg,diag}$ and $\ft\oplus\ft^*$ over $\ft\oplus\ft^*/W$.

A central point of this paper is that for many purposes, one should use the {\em symbolic powers} 
\begin{equation}
\label{eq:symb}
I^{(d)}:=\bigcap_{\alpha\in \Phi^+} \langle \alpha^\vee,\alpha\rangle^d
\end{equation} and define
\begin{definition}
\label{def:symb}
    $$X_{\fg,symb}:=\Proj \bigoplus_{d\geq 0}\e (\Delta^dI^{(d)})$$
\end{definition}
More geometrically, since $I$ is a radical ideal, $I^{(d)}$ is the ideal of functions on $\ft\oplus\ft^*$ that vanishing to multiplicity $d$ along the locus defined by $I$. 
Now $X_{\fg,symb}$ is the {\em symbolic} blow-up of $\ft\oplus \ft^*/W$ along the diagonals. In the trigonometric setting, which we review in Section \ref{sec:coulomb}, this was observed in \cite{GKO}.
\begin{remark}
    When $\fg$ is of type $A$, as mentioned in Remark \ref{remark:haiman}, it follows from Haiman's results that $I^{(d)}=I^d=J^d$.
\end{remark}
For general $\fg$, one of the immediate strengths of Definition \ref{def:symb} is the following:
\begin{theorem}
\label{thm:normality}
$X_{\fg, symb}$ is normal. The same holds for the isospectral variety $Y_{\fg,symb}$.
\end{theorem}
\begin{proof}
The proof is a direct adaptation of \cite[Theorem 4.7.]{Kivinen} to the rational case (see also \cite[Corollary 3.12.]{GKO}). More precisely, $\alpha,\alpha^\vee$ is a regular sequence in $\C[\ft\oplus\ft^*]$, so $I^{(d)}$ is integrally closed for all $d$ and so is the corresponding symbolic Rees algebra. This proves $Y_{\fg,symb}$ is normal. Taking the quotient by the $W$-action preserves normality, so $X_{\fg,symb}$ is also normal.
\end{proof}
Again, we note that from the inclusions $J^d\subseteq I^d\subseteq I^{(d)}$ it follows that there are maps $$X_{\fg,symb}\to X_{\fg, diag}\to X_{\fg, sgn}$$ as well as 
$$Y_{\fg,symb}\to Y_{\fg, diag}\to Y_{\fg, sgn}$$
Unless these are isomorphisms, it is not a priori easy to decide whether $X_{\fg, diag}$ or $X_{\fg,sgn}$ are normal. We also note that it is neither obvious nor relevant for us that the symbolic Rees algebra we consider is finitely generated. In the trigonometric setup, this is ensured by \cite[Remark 1.16]{GKO}. Nevertheless, combining Theorem \ref{thm:zalgebra} and \cite[Proposition 5.2]{BPW} shows that $\e(\Delta^d I^{(d)})$ is the space of sections of a unique line bundle on $X_{\fg,symb}$.

\section{Rational Cherednik Algebras}
\label{sec:cherednik}
In this section, we provide what is probably the most compelling evidence that $X_{\fg,symb}$ is the ``right'' variety to consider. The main result is Theorem \ref{thm:zalgebra}, which is an analog of \cite[Proposition 1.7.]{GS1}. In particular, we have an analog for all Lie types of the main theorem of Gordon--Stafford \cite[Theorem 1.4.]{GS1}.

In this section, we fix once and for all the Lie algebra $\fg$.
Let $c: S\to \C$ be a conjugation-invariant function on the set of reflections $S$ of $W$. We denote its values by subscripts, such as $c_s:=c(s)$. Unless otherwise stated, we take $c$ to be constant. We let $H_c$ be the rational Cherednik algebra associated to the pair $(\ft,W)$. More precisely, following \cite{BEG}, we have
\begin{definition}
Fix $c$ as above. The {\em rational Cherednik algebra} $H_c$ of $\fg$ is the $\C$-algebra generated by 
$\ft, \ft^*$ and $W$ with the relations $wxw^{-1}=w(x), wyw^{-1}=w(y)$ for all $y\in \ft, x\in \ft^*,w\in W$, $[y,x]=\langle y,x \rangle-\sum_{s\in S}c_s\langle y,\alpha_s\rangle\langle\alpha^\vee_s,x\rangle s$ and requiring the algebras generated by $\ft$ and $\ft^*$ to be commutative subalgebras.
\end{definition}

We note that $H_c$ is the specialization of $H_{\hbar, c}$ at $\hbar=1$, where $H_{\hbar,c}$ is the $\C[\hbar]$-algebra with the same generators and relations as above, except that we modify the last relation to read $$[y,x]=\hbar\langle y,x \rangle-\sum_{s\in S}c_s\langle y,\alpha_s\rangle\langle\alpha^\vee_s,x\rangle s.$$

Some important facts about $H_c$ are as follows:
\begin{proposition}
\label{prop:cherednikfacts}
\begin{enumerate}
    \item The algebra $H_{c,\hbar}$ has a triangular decomposition: the multiplication map $$\C[\ft]\otimes \C[W]\otimes \C[\ft^*]\otimes \C[\hbar] \to H_{c,\hbar}$$ is an isomorphism of vector spaces. In particular, for all $c$ we have $$\C[\ft]\otimes \C[W]\otimes \C[\ft^*]\xrightarrow{\cong} H_{c}$$
    \item Let $\deg(x)=\deg(y)=1$ for all $y\in\ft,x\in\ft^*$ and $\deg(w)=0$ for $w\in W$. This defines the {\em degree filtration} on $H_c$ whose associated graded equals the smash-product algebra $\C[\ft\oplus\ft^*]\rtimes W$. 
    \item Let $ord(x)=ord(w)=0,ord(y)=1$. This defines another filtration on $H_c$, called the {\em order filtration}. The associated graded is still as above. We denote associated graded objects for this filtration by $ogr$.
    \item We have injections $\iota_c: H_c\hookrightarrow \Diff(\ft^{reg})\rtimes W$ for all $c$. This is known as the Dunkl embedding in the literature. Moreover, $\iota_c$ become isomorphisms after inverting $\Delta\in H_c$. If we equip $H_c$ with the order filtration and $\Diff(\ft^{reg})\rtimes W$ with the filtration induced by the order of differential operators, the injection $\iota_c$ is also filtration-preserving.
    \item Let $1$ be the map $S\to \C$ taking all reflections to 1 and write $\e=|W|^{-1}\sum_{w\in W}w$, $\e_-=|W|^{-1}\sum_{w\in W}(-1)^{\ell(w)}w$ for the symmetrizing and antisymmetrizing idempotents. Then if $c-1$ is spherical in the sense that $\e H_{c-1}\e$ is Morita equivalent to $H_{c-1}$, we have algebra isomorphisms $\e H_{c-1}\e\cong \e_-H_c\e_-$.
\end{enumerate}
\end{proposition}

In the papers \cite{GS1,GS2}, Gordon and Stafford show that when $\fg=\fg\fl_n$ or $\fs\fl_n$, the representation theory of $H_c$ is closely connected to the geometry of $\Hilb^n(\C^2)$. More precisely, they show using the formalism of $\Z$-algebras that when $\fg=\fg\fl_n$, the spherical subalgebra $U_c=\e H_c \e\subseteq H_c$ is in a precise sense a non-commutative deformation of $\Hilb^n(\C^2)$. Many generalizations of this result have since appeared; the most well-understood cases seem to be the wreath product groups $G(\ell,1,n)$, see, for example, \cite{GordonO, Musson, Pre} and \cite{GordonICM} for further discussion. We will now define similar $\Z$-algebras for any $H_c$ as above.

Following \cite{GS1}, we will write $U_c=\e H_c \e$ for the spherical subalgebra. Using item (5) of Proposition \ref{prop:cherednikfacts}, when $c$ is constant, we can define $U_{c+1}-U_{c}$-bimodules 
$Q_c^{c+1}:=\e H_{c+1}\e_-\Delta$, where $\Delta=\prod_{\alpha\in\Phi^+}\alpha$ as before. 

The maps $\iota_c$ restrict to injections
 $\iota_c: Q_c^{c+1}\hookrightarrow \Diff(\ft^{reg})\rtimes W$ for all $c$. Using the multiplicative structure on the target, one can define $U_{c+i}-U_{c+j}$-bimodules
$$B_{c+i\leftarrow c+j}=Q_{c+i-1}^{c+i}Q_{c+i-2}^{c+i-1}\cdots Q_{c+j}^{c+j+1}.$$ 
For convenience, we fix $c$ and denote these by $B_{ij}$. 
When the parameter $c$ is spherical, we have $B_{ij}\cong Q_{c+i-1}^{c+i}\otimes_{U_{c+i-1}}\cdots\otimes_{U_{c+j+1}}Q_{c+j}^{c+j+1}$, as noted in \cite[(6.3.2)]{GS1}. We will relate the $\Z$-algebra constructed using these bimodules to $X_{\fg,symb}$. To do so, we also need an auxiliary object. Following \cite[\S6.5]{GS1}, we define the $(U_{c+k},H_c)$–bimodules $N_k:= B_{k0}\e H_c$, $k\geq 0$. These also inherit order filtrations from the ambient ring.
The main theorem of this section is the following.

\begin{theorem}\label{thm:zalgebra}
   Let $W$ be a Weyl group with reflection representation $\ft$. 
    Let $J\subset \C[\ft\oplus \ft^*]$ be the ideal generated by $W$-alternating polynomials, and let $I^{(k)}\subset \C[\ft\oplus \ft^*]$ be the ideal defined in Eq. \eqref{eq:symb}. 

Let $c$ be spherical and constant, and $ogr N_k$ be the associated graded of $N_k$ under the order filtration.
    
    Then:
    \begin{enumerate}
    \item There is a natural injective graded $\C[\ft\oplus \ft^*]^W$–linear homomorphism
    $$ J^k\Delta^k\ \hookrightarrow\ ogr N_k.$$
    \item As graded submodules of $\C[\ft\oplus \ft^*]$, one has
    $$ogr N_k\ =\ I^{(k)}\cdot\Delta^k.$$
    \end{enumerate}
    In particular, $J^k\Delta^k\subseteq I^{(k)}\Delta^k=ogr N_k$, and in type~$A$ where $I^{(k)}=J^k$ this recovers the equality $ogr N_k=J^k\Delta^k$ of \cite[Proposition 6.5]{GS1}.
\end{theorem}
\begin{proof}
    (1) The existence of the injective map $J^k\Delta^k\hookrightarrow ogr N_k$ follows exactly as in \cite[\S6.9]{GS1}. Namely, with our assumptions on $c$, 
    \cite[Lemma 6.7(1)]{GS1}  shows that 
    $$J^k=(A^{k}\Delta^k)\C[\ft\oplus \ft^*]=(ogr B_{k,k-1})\cdots(ogr B_{10})(ogr \e H_c)\hookrightarrow ogr N_k.$$
    
    (2) The proof of the second part will occupy the rest of this section. The strategy is a standard reduction to rank one. We begin with an analysis of that case. Here, $W=S_2=\{1,s\}$ and it acts on $\ft\cong\C$ by $s(x)=-x$. Furthermore, $J^k=I^k=I^{(k)}$.  We claim that the map $J^k\Delta^k\hookrightarrow ogr N_k$ is surjective.

The action on $\ft\oplus\ft^*\cong \C^2$ is $(x,y)\mapsto(-x,-y)$. We have $\Delta=x$, $A=\text{span}\{x^ay^b|a+b\text{ is odd}\}$, and $J=(x,y)$. The Dunkl embedding $H_c\hookrightarrow D(\ft^{reg})\rtimes S_2$ sends $y\in \C[\ft^*]$ to $D_y=\partial_x -\frac{c}{x}(1-s)$, which in the associated graded becomes just $y$. In particular, each $x^{i}y^j x^k\in ogr N_k$ with $i+j=k$ appears as the image of $x^iD_y^k x^k$. Since $ogr N_k$ is generated by these elements, we are done.

     For $\alpha\in \ft^*$, denote by $H_\alpha\subset \ft$ the corresponding root hyperplane. 
     Let $U_\alpha=\ft\backslash \bigcup_{\alpha\neq \beta} H_\beta$. If we write $\Delta_\alpha:=\Delta\alpha^{-1}$, then $\C[U_\alpha]\cong \C[\ft][\Delta_\alpha^{-1}]$.

For $p\in H_\alpha\backslash \bigcup_{\beta \neq \alpha}$ the stabilizer $W_p=W_\alpha=\{1,s_\alpha\}\subset W$ is just the reflection along $H_\alpha$. Restricting the representation of $W$ on $\ft\oplus \ft^*$ to $W_\alpha$, we have
    \[
    (\ft\oplus\ft^*) \ \cong\ V\oplus \C^{2(n-1)},
    \]
    where $W_\alpha$ acts on $V:=\C^2$ by $s_\alpha(x,y)=-(x,y)$  and $\C^{2(n-1)}$ is fixed (here, $n=\dim \ft$). Accordingly, we have an isomorphism of algebras
\begin{equation}
\label{eq:slice}
    \C[\ft\oplus \ft^*]|_{U_\alpha\times\ft^*} \;\cong\;     \C[x_\alpha,y_\alpha]\otimes \C[z_1,\dots,z_{2(n-1)}]|_{U_\alpha\times\ft^*},
\end{equation}
where $W_\alpha$ acts trivially on the $z$-variables.
    
In this localization, we have
    $$I^{(k)}|_{U_\alpha\times\ft^*} = J^k|_{U_\alpha\times\ft^*}=(x_\alpha,y_\alpha)^k\cdot \C[\ft\oplus \ft^*]|_{U_\alpha\times\ft^*}$$
and similarly
    $$I^{(k)}\Delta^k|_{U_\alpha\times\ft^*} = J^k\Delta^k|_{U_\alpha\times\ft^*}=(x_\alpha,y_\alpha)^k\Delta^k\cdot \C[\ft\oplus \ft^*]|_{U_\alpha\times\ft^*}.$$
    
On the other hand, the restriction of the $\C[\ft\oplus \ft^*]$-module $ogr N_k$ to $U_\alpha\times\ft^*$ is governed by the rank-one Cherednik algebra for $W_\alpha$ on the slice $V$, tensored with a localized polynomial algebra on the transverse directions. More precisely, Eq. \eqref{eq:slice} together with the PBW theorem identifies
    \[
    (ogr N_k)|_{U_\alpha\times\ft^*} \ \cong\
    (ogr N^{\mathrm{rk}=1}_k)\ \otimes_\C\ \C[z_1,\dots,z_{2(n-1)}]|_{U_\alpha\times\ft^*},
    \]
    where $ogr N^{\mathrm{rk}=1}_k$ is the associated graded of the corresponding module for the rank-one rational Cherednik algebra of $W_\alpha$.

    Thus on each $U_\alpha\times\ft^*$ we have
    \begin{equation}
    \label{eq:localizedequality}
    (ogr N_k)|_{U_\alpha\times\ft^*}
    \ =\
    I^{(k)}\Delta^k|_{U_\alpha\times\ft^*}.
    \end{equation}
    
    By Eq. \eqref{eq:localizedequality}, the $\C[\ft\oplus \ft^*]$-modules $ogr N_k$ and $I^{(k)}\Delta^k$ agree on $\bigcup_\alpha U_\alpha\times \ft^*$, an open set of $\ft\times\ft^*$ whose complement is codimension $\geq 2$. 
    Since both are torsion-free modules of rank 1, they are determined by their localizations at height 1 primes and therefore have to be equal by the algebraic Hartogs' theorem.

\end{proof}
    
\begin{remark}
\label{rem:GS-comparison}
When $W=S_n$, $I^{(k)}=J^k$ via Haiman's results, and the above theorem recovers the identification of the associated graded of $N(k)$ with $J^k\delta^k$ in \cite[Proposition~6.18]{GS1}. 
\end{remark}
As a corollary of the proof, we have the following Gordon--Stafford type localization result.
\begin{corollary}
Let $c$ be constant and spherical. There exists a filtered $\Z$-algebra $B_c$ such that
\begin{enumerate}
\item There is an equivalence of categories from 
$H_c$-mod to $QCoh B_c$.
\item The associated graded algebra $gr B_c$ is isomorphic to the $\Z$-algebra associated with the graded ring $$\bigoplus_{d\geq 0} \e (\Delta^d I^{(d)}),$$ or in other words the homogeneous coordinate ring of $X_{\fg, symb}$.
\item In particular, there is a functor $H_c-\text{filtmod}\to QCoh(X_{\fg,symb})$ from $H_c$-modules with a good filtration to quasicoherent sheaves on $X_{\fg,symb}$.
\end{enumerate}
\end{corollary}
\begin{proof}
The second part follows from Theorem \ref{thm:zalgebra}. The first part follows from \cite[Lemma 5.5]{GS1}, and the third part is standard.
\end{proof}


Analogously to \cite[6.5. and Corollary 6.22.]{GS1} we also have the following theorem:
\begin{theorem}
$ogr \bigoplus B_{k0}\e H_c\cong \bigoplus \Delta^k I^{(k)}$. Therefore the sheaf associated to the spherical representation $\e H_c$ is exactly $\rho_*\cO_{Y_{\fg,symb}}$, the ``Procesi sheaf" for the symbolic blowup.
\end{theorem}
\begin{remark}
Despite the natural construction, we do not know whether $\rho_*\cO_{Y_{\fg,symb}}$ satisfies the properties to be a Procesi sheaf in the sense of Losev \cite{Losev}.
\end{remark}
\begin{remark}
When $c=1/h$ where $h$ is the Coxeter number, $H_c$ possesses a one-dimensional representation \cite{BEG}. By a similar argument as in \cite{GS2}, there is a surjection from $ \e I^{(i)}/\e \fm I^{(i)}\twoheadrightarrow \gr \e L_{c+i}$, where the associated graded is for the tensor product filtration arising from $\e L_{c+ih}\cong B_{i0}\otimes_{U_c} \e L_{c}$. When $i=1$, $\gr \e L_{c+1}$ is Gordon's canonical quotient. We infer that there is a surjection from the ring of functions on the ``punctual'' part of $X_{\fg,symb}$ to this quotient.
\end{remark}

\section{Coulomb branches}
\label{sec:coulomb}
We recall a natural incarnation of the trigonometric version of $X_\fg$ following \cite{GKO}. Let $I_G\subseteq \C[T^*T^\vee]$ be the ideal $$I_G=\bigcap_\alpha \langle 1-e^{\alpha}, \alpha^\vee\rangle.$$
We denote these versions by $X_G$ and in particular make the following definition, with obvious versions for ``$sgn, diag$'' in place of ``$symb$''.
\begin{definition}
    $$X_{G,symb}=\Proj \bigoplus_{d\geq 0} ({\bf e} I_G^{(d)})$$
\end{definition}
We then consider the partially resolved Coulomb branch for the $(G, Ad)$-theory constructed in \cite[Section 3]{GKO}. More precisely, the ordinary/naive Coulomb branch algebra from \cite{BFN} for a pair $(G,N), N\in Rep(G)$ is constructed as the spectrum of a commutative ring, which in turn is constructed using a convolution product in the $G[[t]]$-equivariant Borel--Moore homology of the space of triples 
$$\cR=\{[g,s]\in G((t))\times^{G[[t]]} N[[t]]|gs\in N[[t]]\}$$ 
In \cite{GKO}, this ind-scheme is replaced by $$\cR^d=\{[g,s]\in G((t))\times^{G[[t]]} N[[t]]|gs\in t^{-d}N[[t]]\}$$ and a convolution product is constructed on the space 
$$\bigoplus_{d\geq 0} H^{G[[t]]}_*(\cR^d)$$
When $N=Ad,$ we denote the Proj of this algebra as $X_{G,Coulomb}$.

From \cite[Theorem 3.8.]{GKO} we have that
\begin{theorem}
We have an isomorphism $X_{G,Coulomb}\cong X_{G,symb}$.
\end{theorem}
\begin{remark}
    There is an obvious generalization of this result to $K$-theoretic Coulomb branches, identifying the Proj of $$\bigoplus_{d\geq 0} K^{G[[t]]}(\cR^d)$$ with the symbolic blow-up in $\bigcap_{\alpha\in \Phi^+} \langle 1-e^{\alpha},1-e^{\alpha^\vee}\rangle\subseteq \C[T\times T^\vee]$.
\end{remark}
\begin{remark}
Based on the trigonometric result in \cite[Theorem 3.6]{GKO}, we also expect in the rational case studied in this paper that
$\e \Delta I=\e \Delta J=\Delta A$ for all $\fg$. 

\end{remark}

The upshot of the construction of $X_G$ as a Coulomb branch is that it comes with a natural quantization. Indeed, as studied there, $T^*T^\vee/W$ quantizes to a spherical trigonometric Cherednik algebra and the partial resolution to a certain $\Z$-algebra built out of this. Moreover, the interpretation as a Hamiltonian reduction as explained e.g. in \cite[Remark 5]{Weekes} gives $X_{G,symb}$ a natural Poisson structure. By \cite[Theorem 1.1. and Lemma 2.1.]{BellamySymplectic}, we have the following.
\begin{proposition}
\label{prop:sympsingtrig}
$X_{G,symb}$ has symplectic singularities.
\end{proposition}

\section{Geometric properties of $X_\fg$}
\label{sec:geomprop}
To understand the eventual applications to coherent sheaves on the varieties $X_{\fg}$, one needs a better handle on their geometry. For $\Hilb^n(\C^2)$, many results use the modular interpretation, which is not available for general $\fg$.

First, we have an analog of Proposition \ref{prop:sympsingtrig} for $X_{\fg,symb}$. 
\begin{proposition}
\label{prop:sympsingrat}
The variety $X_{\fg,symb}$ has symplectic singularities.
\end{proposition}
\begin{proof}
Recall from \cite[Section 3]{GKO} and \cite{Etingof} that locally in $T^*T^\vee/W$, the singularities are modeled by $\ft\oplus \ft^*/W_{\ft}$ for various Cartan subalgebras of centralizers of semisimple elements $a\in G$ with maximal semisimple rank. More precisely, as shown there, there are isomorphisms of formal neighborhoods $(T^*T^\vee/W)^{\wedge (0,a)}\cong (\ft\oplus \ft^*/W_{\ft})^{(0,0)}$ and these respect the quantizations to trigonometric, resp. rational, Cherednik algebras. This is an incarnation of Borel--de Siebenthal theory, and the $\ft$ that appear can be read off by removing a single vertex from the affine Dynkin diagram of $G$. So formally locally, the singularities of $X_{\fg,symb}$ are the same as those of $X_{G,symb}$.
By Artin approximation, since we work with varieties (of finite type) over $\C$, we can extend the formal local isomorphism to \'etale local ones.
As the singularities of $X_{G,symb}$ are symplectic, a fortiori so are those of $X_{\fg, symb}$.
 \end{proof}

Finally, we note that $X_{\fg,symb}$ is only smooth in type A. Indeed, as shown above, $X_{\fg,symb}$ has symplectic singularities, and smoothness would imply that $X_{\fg,symb}\to \ft\oplus \ft^*/W$ is a symplectic resolution. By Kaledin \cite{Kaledin}, this can only happen in types ABC. However, in types BC the variety $X_{\fg,symb}$ is singular by construction, as can be seen from the comparison to quiver varieties in Proposition \ref{prop:quiver}. The above results nonetheless suggest that $X_{\fg,symb}$ has appeared in disguise before. We discuss these connections next.

In the simply laced cases, we have
\begin{theorem}
\label{thm:qfact}
$X_{\fg,symb}$ is a $\Q$-factorial terminalization of $\ft\oplus \ft^*/W$. 
\end{theorem}
\begin{proof}
Since $X_{\fg,symb}$ is normal by Theorem \ref{thm:normality}, conical, and has symplectic singularities by Proposition \ref{prop:sympsingrat}, we may choose a $\Q$-factorial terminalization $\widetilde{X}\to \ft\oplus\ft^*/W$ so that we have an intermediate symplectic partial resolution 
$$\widetilde{X}\to X_{\fg,symb}\to \ft\oplus\ft^*/W.$$ As explained in \cite[Proof of Proposition 2.1]{BAL}, $X_{\fg,symb}$ corresponds to a face of the ample cone of $\widetilde{X}$. Since $\fg$ is simply laced, by, for example, \cite{BellamyCounts}, $H^2(\widetilde{X},\Q)\cong \Q$. Since $X_{\fg,symb}$ is nontrivial, the first map must be an isomorphism.
\end{proof}

Based on the theorem, we may think of $X_{\fg,symb}$ as a more explicit construction of the implicitly chosen partial symplectic resolution corresponding to the filtered quantization of $\ft\oplus\ft^*/W$ appearing in \cite{Losev,BAL}. Denoting it by $X_{\fg,Losev}$, the theorem says $X_{\fg,Losev}\cong X_{\fg,symb}$ for simply laced types.

For more general $\fg$, one can always choose a $\Q$-factorial terminalization $$\pi: \widetilde{X}\to \ft\oplus \ft^*/W$$ as above. The finitely many possible choices have been classified in \cite{BellamyCounts}.
For example, in types $B_n$ and $C_n$, the variety $\widetilde{X}$ can be constructed using quiver varieties or $(\Z/2\Z)$-Hilbert schemes. 

Similarly to the proof of Theorem \ref{thm:qfact}, partial resolutions of $\ft\oplus \ft^*/W$ can be constructed using birational contractions of $\widetilde{X}$, which match filtered quantizations of $\ft\oplus \ft^*/W$.
For non-simply laced $\fg$, denote the exceptional divisor of $\pi$ by the letter $D$. Then, as explained in \cite[Proposition 2.1.]{BAL}, $D=\sum_{s\in \text{Refl}/\sim} D_s$, where $D_s$ are irreducible components of $D$ corresponding to conjugacy classes of simple reflections in $W$. By $\Q$-factoriality, $\ell D$ is Cartier for some $\ell>0$. One can choose $\widetilde{X}$ so that $\ell D$ lies in the closure of the ample cone. By the contraction theorem, there is a unique intermediate (normal, symplectic) partial resolution $\widetilde{X}\to X\to \ft\oplus \ft^*/W$ for which $\cO(\ell D)$ on $\widetilde{X}$ is lifted from an ample line bundle on $X$. We refer to the proof of \cite[Proposition 2.1.]{BAL} for details. We take $X_{\fg,Losev}=X$.

The above construction has the obvious pitfall of being inexplicit. However, 
we conjecture in general that
\begin{conjecture}
\label{conj:losev}
For a well-chosen $\widetilde{X}$, we have an isomorphism 
$$X_{\fg,symb}\cong X_{\fg,Losev}$$
\end{conjecture}
In simply laced types and types BC, this is true by the above as well as \cite[Remark 2.2.]{BAL} and Proposition \ref{prop:quiver}. The only missing cases are then $G_2$ and $F_4$, which would follow from a comparison of the periods of the quantizations of $X_{\fg,symb}$ and $X_{\fg,Losev}$.

\subsection{Calogero--Moser spaces}
Consider the two-sided cells of $W$. These are the equivalence classes for a certain equivalence relation on $W$, defined using the Iwahori-Hecke algebra of $W$ with equal parameters. Denoting $\mathfrak{C}_W=\{\text{Two-sided cells in } W\}$, there is a map $\mathbf{a}: \mathfrak{C}_W\to \N$, commonly called Lusztig's $\mathbf{a}$-function. Moreover, we denote by $\mathbf{A}:\mathfrak{C}_W\to\N$ the function $\mathbf{A}(\lambda)=\mathbf{a}(w_0\lambda)$.

As a more refined version of Conjecture \ref{conj:bonnafe}, we expect
\begin{conjecture}
\label{conj:fixedpoints}
There is a bijection
$$X_{\fg,symb}^{(\C^*)^2}\leftrightarrow \mathfrak{C}_W$$
If $\delta_\lambda$ denotes the skyscraper sheaf at the corresponding fixed point, in the equivariant $K-$group $K^0_{(\C^*)^2}(X_{\fg,symb})$ we have 
$$[\delta_\lambda\otimes \cO(1)]=q^{\mathbf{a}(\lambda)}t^{\mathbf{A}(\lambda)}[\delta_\lambda]$$
\end{conjecture}
 For type $A$, this follows from the usual bijection between fixed points in the Hilbert scheme and partitions. We prove the first part of the conjecture for types $BC$ in the appendix.
Conjecture \ref{conj:fixedpoints} fits well with a conjecture of Bonnaf\'e and Rouquier \cite{BonnafeRouquier}, who have conjectured the existence of a natural bijection between fixed points in Calogero--Moser spaces and two-sided cells. More precisely, we define
\begin{definition}
\label{def:CM}
The Calogero--Moser space of $\fg$ is 
$$\CM_{c}=\Spec Z(H^{rat}_{c,0})$$ where $H_{c,\hbar}$ is the rational Cherednik algebra of $\fg$ from Section \ref{sec:cherednik} and $Z(-)$ denotes taking the center.
\end{definition}
The conjecture we want to compare to is as follows. When $c$ is constant and generic enough, there is a bijection
$$\CM_c^{\C^*}\leftrightarrow \mathfrak{C}_W,$$ where $\C^*$ is the Hamiltonian $\C^*$-action on $\CM_c$.
In fact, it is possible to define certain {\em Calogero--Moser cells} in $W$ using the $\C^*$-fixed points in $\CM_c$. The conjecture of \cite{BonnafeRouquier} is that for Coxeter groups these coincide with the two-sided cells defined by Lusztig.

For types other than $ABC$, Bellamy has shown \cite{Bellamy} that $\CM_c$ is singular for all values of $c$.
Based on the type $ABC$ cases where $X_\fg$ admits a quiver variety description, we conjecture:
\begin{conjecture}
\label{conj:hyperkaehler}
    For any $\fg$, the varieties $\CM_c$ and $X_{\fg, symb}$ admit hyper-Kähler structures such that there is a $U(1)$-equivariant homeomorphism $\CM_c\to X_{\fg,symb}$, given by rotating the complex structure (see, for example, \cite[3.7.]{GordonO}). In particular, there is a bijection $(X_{\fg,symb})^{\C^*}\leftrightarrow \CM_c^{\C^*}$, where $\C^*$ is the Hamiltonian $\C^*\subseteq (\C^*)^2$ on the left.
\end{conjecture}

\subsection{Further conjectures}
In this section, we make some further speculation on $X_{\fg,symb}$.

There is another reason to believe Conjecture \ref{conj:fixedpoints}. It might seem contrived but was the original motivation for that conjecture before the author learned of Conjecture \ref{conj:bonnafe}. As explained in \cite{KT}, starting from an elliptic regular semisimple element in the loop Lie algebra of $\fg\fl_n$, i.e., $\gamma\in \fg\fl_n((t))$, one can construct a sheaf $\cF_\gamma$ on $\Hilb^n(\C^2)$ supported on the punctual Hilbert scheme. It is quasi-coherent and $(\C^*)^2$-equivariant. Conjecturally, $\cF_\gamma$ is coherent and the localization theorem in $K$-theory lets one write the $K$-class $[\cF_\gamma]=\sum_\lambda a_\lambda(q,t) \tilde{H}_\lambda$, where $\tilde{H}_\lambda$ are the fixed point classes and $a_\lambda(q,t)$ are some rational functions in $\C(q,t)$. If one believes the triangle of connections from the introduction, as explained in \cite{KT}, the functions $a_{\lambda}(q,t)$ are closely related to Shalika germs for the element $\gamma$.

The main construction in \cite{GKO} (which is a precursor to the last section of \cite{KT}) extends to other groups, as discussed in the introduction. For an elliptic regular semisimple element $\gamma\in \fg((t))$, the pushforward of the sheaf $\cF_\gamma$ to the naive Coulomb branch $T^*T^\vee/W$ is supported at finitely many points $(0,b)\in T^*T^\vee/W$, corresponding to the endoscopic decomposition of cohomology. As explained in \cite[Theorem 2.16.]{GKO}, the stable part of cohomology corresponds to $(0,1)$. 

Near this singularity, $T^*T^\vee/W$ is formally locally isomorphic to $\ft\oplus \ft^*/W$, so pulling back $X_{G,Coulomb}$ and the sheaf $\cF_\gamma$ along the inclusion of the formal neighborhood, one gets a (quasi-)coherent sheaf $\cF_\gamma^{rat}$ on $X_{\fg,symb}$ above the formal neighborhood $(\ft\oplus \ft^*/W)^{\wedge 0}$. Using the perverse filtration for the BM homology of affine Springer fibers, it should be possible to upgrade this construction to be equivariant for the $(\C^*)^2$-action on $\ft\oplus \ft^*/W$. Assuming this, we expect that the $K$-theory localization of $\cF_\gamma^{rat}$ to the fixed points on $X_{\fg,rat}$, which can be written as $$[\cF_\gamma]=\sum_x a_\lambda(q,t) [\delta_x],$$ encodes information about the Shalika germs of $\gamma$. Even without knowing which $x$ appear, conjectures of Assem \cite{Assem} suggest that the stable Shalika expansion only contains special nilpotent orbits in $\fg((t))$. These are also in bijection with two-sided cells in $W$.

The last conjecture we make is about the braid invariants. Recall that in \cite[Theorem 1.5.]{GH} it was shown that the $y$-ified or monodromically deformed version of the Hochschild homology of the Rouquier complex of the $k$th power of the full twist braid in $Br_n$ is isomorphic as a doubly graded $\C[x_1,\ldots, x_n,y_1,\ldots, y_n]$-module to $J^k=I^{(k)}$. The proof of that theorem heavily uses the earlier computation of the ordinary Hochschild homology of this Rouquier complex by Elias and Hogancamp in \cite{EH}, in particular the fact that this homological invariant is {\em parity}, or in other words that it is only supported in even homological degrees.

For the braid groups of other types, it should still be true that the Hochschild homology of the Rouquier complex of the full twist is parity. The only evidence we have for this is the purity of the corresponding affine Springer fibers. Though we do not do it here, it is also possible to define ``link-splitting maps'' as in Section 4 of \cite{GH} for other $W$ and prove analogs of their flatness properties. Using these properties and the conjectural parity of the $k$:th powers of the full twists, one can show that in general this monodromically deformed invariant is isomorphic as a doubly graded $\C[\ft\oplus \ft^*]$-module to $I^{(k)}$. We leave this question open.

\appendix
\section{Proof of Conjectures \ref{conj:fixedpoints}, \ref{conj:hyperkaehler} in types $B_n, C_n$}
\label{app:proofs}
\begin{proposition}
\label{prop:quiver}
Let $\fg=\fs\fo_{2n+1}$ or $\fs\fp_{2n}$. Then $X_{\fg,symb}\cong M_\theta(Q)$, where $Q$ is the cyclic quiver with two vertices and dimension vector $(n,n)$ together with a dimension $1$ framing at the extending vertex, and $\theta$ is a generic stability condition on the wall containing $(0,1)$. This identification respects the torus actions.
\end{proposition}
\begin{proof}
This follows from Theorem \ref{thm:zalgebra}(2), \cite[Theorem 4.1.]{GordonO} and the uniqueness of quantizations of line bundles in \cite[Proposition 5.2]{BPW}. In the notation of Gordon, we have $c_1=2h$ and $c_\sigma=H_0-H_1=-2H_1$. Since $c_1=c_\sigma=c$, i.e., we are in the equal parameter case, the stability condition $\theta$ in loc. cit. becomes $(0,-2c)$, because $-h-H_1=0$ and $H_1=-2c$. Since the quiver varieties are isomorphic under negating the stability condition, we see we are on the positive vertical wall and therefore have the desired isomorphism.
\end{proof}
\begin{corollary}
Conjectures \ref{conj:losev} and \ref{conj:hyperkaehler} are true in types $B_n, C_n$.
\end{corollary}
\begin{proof}
The first one is \cite[Remark 2.2.]{BAL}, and the second one follows from \cite[3.5.--3.7.]{GordonO}.
\end{proof}

For the second result we need to set up some notation. Recall the core-and-quotient construction establishing a bijection between $\ell$-multipartitions of $n$ and partitions of $n\ell+r$ with a fixed $\ell$-core partition $s$ of size $r$. This gives an injective map $$\tau_s: P(\ell,n)\to P(n\ell+r)$$
In our case $r=0$ and $\ell=2$, in which case we write $\tau_s$ as a bijection 
$$\tau: P(2,n)\to \{\lambda\in P(2n)|\lambda \text{ has trivial $2$-core}\}$$
Given a partition $\lambda$ of any size and $\ell\geq 1$, recall that for $0\leq i\leq \ell-1$ a box in the Young diagram of $\lambda$ is called $i$-addable (resp. $i$-removable) if its content is congruent to $i$ mod $\ell$ and it is addable (resp. removable). Given a subset $J\subseteq \{0,\ldots,\ell-1\}$, we call the $J$-heart of $\lambda$ the partition obtained by repeatedly removing all $i$-removable boxes for all $i\in J$. Having the same $J$-heart sets up an equivalence relation on partitions of varying or fixed size, and partitions these sets into equivalence classes called $J$-classes.
\begin{proposition}
\label{prop:jclasses}
The torus-fixed points $M_\theta(Q)^{(\C^*)^2}$ defined as above are in bijection with $J$-classes in $\tau(P(2,n))$, where $J=\{0\}$.
\end{proposition}
\begin{proof}
This is \cite[Proposition 8.3.]{GordonO}. See also \cite[Corollary 10.4 and Remark 10.5]{Pre}.
\end{proof}
\begin{proposition}
    There is a bijection between $\tau(P(2,n))$ and irreducible representations of $W$, such that the partition of $\tau(P(2,n))$ into $0$-classes induces the partition of $\text{Irr}(W)$ into families. In particular, Conjecture \ref{conj:fixedpoints} is true in this case.
\end{proposition}
\begin{proof}
This follows from \cite[Theorem 3.3.]{GordonMartino} and the results of \cite{Garfinkle}.
\end{proof}

\begin{example}
    We illustrate the bijection of the previous proposition. Recall that in \cite{LusztigBook}, the families can be described using symbols with the same entries.
    A natural bijection between bipartitions and symbols of rank $n$ and defect $1$ is easy to set up, so that a given bipartition $(\lambda,\mu)$ is associated to a symbol as follows:
    If $$\lambda_1\leq\ldots\leq \lambda_{\ell(\lambda)},\;\mu_1\leq\ldots\leq\mu_{\ell(\mu)}$$ are the partitions written in increasing order (possibly empty), if necessary add $k$ zeros to the front of $\lambda$ or $\mu$ so that the resulting pair of nondecreasing sequences of integers $(\lambda',\mu')$ satisfies $\ell(\lambda')=\ell(\lambda)+k=\ell(\mu)+1$ or
    $\ell(\mu')=\ell(\mu)+k=\ell(\lambda)-1$
 depending on whether $\ell(\lambda)\leq \ell(\mu)$ or $\ell(\lambda)>\ell(\mu)$. Now consider the symbol with rows $$
    \begin{bmatrix}
        \lambda_1'+0&&\lambda_2'+1&&\cdots&&\lambda_{\ell(\lambda')}'+\ell(\lambda')-1\\
        &\mu_1'+0&&\cdots&&\mu_{\ell(\mu')}'+\ell(\mu')-1&
    \end{bmatrix}$$
    
    Now let $n=3$. Then $P(2,3)$ consists of $10$ bipartitions. 

We list these, the partitions of $6$ with trivial $2$-core together with their contents mod $2$, as well as the corresponding symbols.
$$\begin{array}{|c|c|c|c|}\hline
\text{Partition}& \text{$0$-Heart} & \text{Bipartition}& \text{Symbol} \\\hline
{\small \young(1,0,1,0,1,0)} &  {\small \young(1,0,1,0,1,0)} &  ({\tiny \emptyset,\yng(1,1,1)}) &  
 \begin{bmatrix}
0&&1&&2&&3\\&1&&2&&3& 
\end{bmatrix}\\\hline
{\small \young(0,1,0,1,01)}&  {\small \young(1,0,1,01)}&  ({\tiny \yng(1,1,1),\emptyset}) &  \begin{bmatrix}
1&&2&&3\\&0&&1& 
\end{bmatrix}  \\\hline
{\small \young(1,0,10,01)}& {\small \young(1,0,1,01)}&  ({\tiny \yng(1),\yng(1,1)}) &    
\begin{bmatrix}
0&&1&&3\\&1&&2& 
\end{bmatrix} 
\\\hline
{\small \young(1,0,1,010)}&  {\small \young(1,0,1,01)}&  ({\tiny \emptyset,\yng(1,2)}) &  
\begin{bmatrix}
0&&1&&2\\&1&&3& 
\end{bmatrix} \\\hline
{\small \young(0,1,0101)}&  {\small \young(1,0101)}&  ({\tiny \yng(1,2),\emptyset}) &  
\begin{bmatrix}
1&&3\\&0& 
\end{bmatrix}\\\hline
{\small \young(10,0101)}&  {\small \young(1,0101)}& ({\tiny \yng(2),\yng(1)}) &
\begin{bmatrix}
0&&3\\&1& 
\end{bmatrix} \\\hline
{\small \young(1,01010)}&  {\small \young(1,0101)}& (\tiny{\emptyset},\yng(3)) &  \begin{bmatrix}
0&&1\\&3& 
\end{bmatrix} \\\hline
{\small \young(101,010)}&  {\small \young(101,010)}&  ({\tiny \yng(1),\yng(2)}) &  
\begin{bmatrix}
0&&2\\&2& 
\end{bmatrix}\\\hline
{\small \young(01,10,01)}&  {\small \young(01,10,01)}& ({\tiny \yng(1,1),\yng(1)}) &  \begin{bmatrix}
1&&2\\&1& 
\end{bmatrix}   \\\hline
{\small \young(101010)}&  {\small \young(10101)}&  ({\tiny \yng(3),\tiny{\emptyset}}) &  \begin{bmatrix}
3\\- 
\end{bmatrix}   \\\hline
\end{array}
$$

\end{example}

\section{Computer calculations}
\label{sec:examples}

We have included the code used for the $G_2$ and $B_3$ calculations below, hoping that they are illustrative of the general procedure. 

\begin{example}
The root system of type $G_2$ can be realized in the plane $x+y+z=0$ inside $\R^3$, with positive simple roots $\beta=2x-y-z$ and $\alpha=x-y$. We can thus realize $W$ as the reflection group generated by the two simple reflections about the corresponding root hyperplanes. In the code below, $M_1$ is the matrix for $s_\alpha$ and $M_2$ is the matrix for $s_{\beta+3\alpha}$.

\begin{lstlisting}
julia> M1=matrix(QQ,[0 1 0;1 0 0; 0 0 1]);

julia> M2=matrix(QQ,[-1//3 2//3 2//3; 2//3 2//3 -1//3; 2//3 -1//3 2//3]);

julia> G=matrix_group([M1,M2]);

julia> M1D=block_diagonal_matrix([M1, transpose(inv(M1))]);M2D=block_diagonal_matrix([M2,transpose(inv(M2))]);GD=matrix_group([M1D,M2D]);

julia> F=abelian_closure(QQ)[1];IR=invariant_ring(GD);R=polynomial_ring(IR);x=gens(R);

julia> chi=Oscar.class_function(GD,[F(det(representative(c))) for c in conjugacy_classes(G)]);

julia> ri=relative_invariants(IR,chi);

julia> J=ideal(R,ri);

julia> I1=ideal(R,[x[1]-x[2],x[4]-x[5]]);I2=ideal(R,[x[1]-x[3],x[4]-x[6]]);I3=ideal(R,[x[2]-x[3],x[5]-x[6]]);I4=ideal(R,[2*x[1]-x[2]-x[3],2*x[4]-x[5]-x[6]]);I5=ideal(R,[2*x[2]-x[1]-x[3],2*x[5]-x[4]-x[6]]);I6=ideal(R,[2*x[3]-x[1]-x[2],2*x[6]-x[4]-x[5]]);I=intersect(I1,I2,I3,I4,I5,I6);

julia> J==I
true
\end{lstlisting}
\end{example}
\begin{example}
In type $B_n$, the root system can be realized in $\R^n$ by $\pm x_i, \pm(x_i\pm x_j), i,j=1,\ldots,n$. Similarly, the type $C_n$ root system can be realized by $\pm 2x_i, \pm (x_i \pm x_j)$
The hyperoctahedral group $W$ acts in the obvious way. We can take $\ft\oplus \ft^*$ to be $\C^{2n}$ with coordinates $x_1,\ldots, x_n,y_1,\ldots, y_n$ and write 
$$I=\cap_i\langle x_i, 2y_i\rangle \bigcap \cap_{i<j} \langle  x_i \pm x_j, y_i\pm y_j\rangle.$$

Below is the code for the $B_3$-case, computing $I$ and $J$. One can check that $I$ has one extra generator in its minimal generating set, which is not $W$-invariant. In particular, one has $J\subsetneq I$ but $\e I=\e J$. 
\begin{lstlisting}
julia> M1=matrix(QQ,[0 1 0; 1 0 0; 0 0 1]);M2=matrix(QQ,[1 0 0; 0 0 1; 0 1 0]);M3=matrix(QQ,[-1 0 0; 0 1 0; 0 0 1]);W=matrix_group([M1,M2,M3]);

julia> M1D=block_diagonal_matrix([M1,transpose(inv(M1))]);M2D=block_diagonal_matrix([M2,transpose(inv(M2))]);M3D=block_diagonal_matrix([M3,transpose(inv(M3))]);WD=matrix_group([M1D,M2D,M3D]);

julia> IR=invariant_ring(WD);R=polynomial_ring(IR);x=gens(R); F=abelian_closure(QQ)[1];

julia> chi=Oscar.class_function(WD,[F(det(representative(c))) for c in conjugacy_classes(W)]);

julia> ri=relative_invariants(IR,chi);J=ideal(R,ri);

julia> I1=ideal(R,[x[1]-x[2],x[4]-x[5]]); I2=ideal(R,[x[1]-x[3],x[4]-x[6]]); I3=ideal(R,[x[2]-x[3],x[5]-x[6]]);I4=ideal(R,[x[1]+x[2],x[4]+x[5]]); I5=ideal(R,[x[1]+x[3],x[4]+x[6]]); I6=ideal(R,[x[2]+x[3],x[5]+x[6]]);I7=ideal(R,[x[1],x[4]]);I8=ideal(R,[x[2],x[5]]);I9=ideal(R,[x[3],x[6]]);

julia> I=intersect(I1,I2,I3,I4,I5,I6,I7,I8,I9);

julia> I==J
false
\end{lstlisting}

\end{example}

\subsection*{Acknowledgments}
The computer experiments in Appendix \ref{sec:examples} were computed with the OSCAR computer algebra system \cite{OSCAR}, with help from Johannes Schmitt and Ulrich Thiel. We also thank Alex Weekes for correspondence, Gwyn Bellamy and C\'edric Bonnaf\'e for comments on a draft of this paper, and Ivan Losev for helpful comments.

\end{document}